\documentclass{amsart}
\usepackage{amsmath}%
\usepackage{amsfonts}%
\usepackage{amssymb}%
\usepackage{graphicx}
\usepackage{color}
\usepackage{enumerate}
\usepackage{verbatim}
\usepackage{hyperref}
\usepackage{enumerate}
\usepackage{eucal}

\usepackage{pgfplots}

\newtheorem{theorem}{Theorem}
\newtheorem*{theorem*}{Theorem}
\theoremstyle{plain}

\newtheorem{claim}{Claim}

\newtheorem{conjecture}{Conjecture}
\newtheorem{corollary}{Corollary}

\newtheorem{definition}{Definition}
\newtheorem*{example}{Example}

\newtheorem{lemma}{Lemma}
\newtheorem*{lemma*}{Lemma}

\newtheorem{proposition}{Proposition}
\newtheorem{remark}{Remark}

\numberwithin{equation}{section}

\newtheorem{remarks}{Remarks}

\newcommand{\R}{\mathbb{R}}

\newcommand{\beeq}{\begin{equation}}
\newcommand{\eneq}{\end{equation}}
\newcommand{\beeqs}{\begin{eqnarray*}}
	\newcommand{\eneqs}{\end{eqnarray*}}
\newcommand{\besp}{\begin{split}}
	\newcommand{\ensp}{\end{split}}
\newcommand{\bepr}{\begin{proof}}
	\newcommand{\enpr}{\end{proof}}
\newcommand{\bethr}{\begin{theorem}}
	\newcommand{\enthr}{\end{theorem}}
\newcommand{\beths}{\begin{theorem*}}
	\newcommand{\enths}{\end{theorem*}}
\newcommand{\becor}{\begin{corollary}}
	\newcommand{\encor}{\end{corollary}}
\newcommand{\bere}{\begin{remark}}
	\newcommand{\enre}{\end{remark}}
\newcommand{\bers}{\begin{remarks}}
	\newcommand{\enrs}{\end{remarks}}
\newcommand{\beres}{\begin{remark*}}
	\newcommand{\enres}{\end{remark*}}
\newcommand{\bele}{\begin{lemma}}
	\newcommand{\enle}{\end{lemma}}
\newcommand{\beles}{\begin{lemma*}}
	\newcommand{\enles}{\end{lemma*}}
\newcommand{\bepro}{\begin{proposition}}
	\newcommand{\enpro}{\end{proposition}}
\newcommand{\bepros}{\begin{proposition*}}
	\newcommand{\enpros}{\end{proposition*}}
\newcommand{\becl}{\begin{claim}}
	\newcommand{\encl}{\end{claim}}
\newcommand{\beex}{\begin{example}}
	\newcommand{\enex}{\end{example}}
\newcommand{\beexs}{\begin{example*}}
	\newcommand{\enexs}{\end{example*}}
\newcommand{\beco}{\begin{conjecture}}
	\newcommand{\enco}{\end{conjecture}}
\newcommand{\becos}{\begin{conjecture*}}
	\newcommand{\encos}{\end{conjecture*}}
\newcommand{\bede}{\begin{definition}}
	\newcommand{\bedes}{\begin{definition*}}

		\textheight=\dimexpr
		\ifcase 10pt
		\or 
		\or 
		\or 
		\or 
		\or 
		\or 
		\or 
		\or 
		57\or 
		52\or 
		48\or 
		44\or 
		41\fi 
		\baselineskip+\topskip\relax
		
		\title[Isoparametric hypersurfaces as initial data for the MCF]{
		Isoparametric hypersurfaces of Riemannian manifolds as initial data for the mean curvature flow}

		\author[Guimar\~{a}es, F.]{Felippe Guimar\~{a}es}
		\address{Felippe Guimar\~{a}es -- KU Leuven, Department of Mathematics \newline \indent Celestijnenlaan 200B – Box 2400, B-3001 Leuven, Belgium}%
		\address{Instituto de Matem\'atica e Estat\'istica, Universidade de S\~ao Paulo\newline \indent Rua do Mat\~ao 1010, S\~ao Paulo SP, 05508-090, Brazil}%
		\email{felippe@impa.br}%
		\thanks{The first author was supported by CAPES - Finance Code 001, during his stay at Universidade \newline \indent de Bras\'ilia, and by FAPESP grants \#2019/19494-0, \#2021/12348-8}
		
		\author[dos Santos, J. B. M.]{Jo\~ao Batista Marques dos Santos}
		\address{Jo\~ao Batista Marques dos Santos -- Universidade de Bras\'{i}lia \newline \indent Campus Universit\'ario Darcy Ribeiro, Brasilia - DF, 70910-900, Brazil}%
		\email{j.b.m.santos@mat.unb.br}%
		\thanks{The second author was supported by Capes and CNPq}

		\author[dos Santos, J. P.]{Jo\~{a}o Paulo dos Santos}
		\address{Jo\~{a}o Paulo dos Santos -- Universidade de Bras\'{i}lia \newline \indent Campus Universit\'ario Darcy Ribeiro, Brasilia - DF, 70910-900, Brazil}%
		\email{joaopsantos@unb.br}%
		\thanks{The third author was supported by FAPDF 0193.001346/2016 and CNPq 315614/2021-8}


\begin{document}
			
			\maketitle

\begin{abstract}
				We show that the evolution of isoparametric hypersurfaces of Riemannian manifolds by the mean curvature flow is given by a reparametrization of the parallel family in short time, as long as the uniqueness of the mean curvature flow holds for the initial data and the corresponding ambient space. As an application, we provide a class of Riemannian manifolds that  admit hypersurfaces with constant principal curvatures, which are not isoparametric hypersurfaces.
				Furthermore, for a class of ambient spaces, we show that the singularities developed by the mean curvature flow with isoparametric hypersurfaces as the initial data are Type I singularities. We apply our results to describe the evolution of isoparametric hypersurfaces by the mean curvature flow in ambient spaces with nonconstant sectional curvature, such as homogenous 3-manifolds $\mathbb{E}(\kappa, \tau)$ with 4-dimensional isometry groups, and Riemannian products $\mathbb{Q}^2_{c_1} \times \mathbb{Q}^2_{c_2}$ of space forms.
			\end{abstract}
			
			\vspace{5mm}

Given a hypersurface $M^n$ of Riemannian manifold $\widetilde{M}^{n+1}$, we say that $M$ evolves by the mean curvature flow (MCF) if there is a time-dependent family of smooth hypersurfaces with $M$ as initial data such that the velocity of the evolution at each point of such family is given by the mean curvature vector field of the correspondent hypersurface at that point. There is an extensive literature on the study of MCF, mainly when the ambient space $\widetilde{M}^{n+1}$ is the Euclidean space $\mathbb{R}^{n+1}$. However, cases where the ambient space is a general Riemannian manifold and when the codimension is greater than one have also been considered recently. We suggest the surveys \cite{ColdingMinicozziKjaer_mcf,surveySmoczyk} and references within for a good overview of the mentioned topics.
			
A hypersurface is said to be isoparametric if its parallel hypersurfaces have constant mean curvature. In \cite{reis-tenenblat} the authors showed that a hypersurface $M^n$ of a space form $\mathbb{Q}^{n+1}_c$ is the initial data for a solution for the MCF given by a reparametrization of the flow of parallel hypersurfaces if and only if $M^n$ is an isoparametric hypersurface. Recall that a space form $\mathbb{Q}^{n+1}_c$ is a complete, simply connected Riemannian manifold with constant sectional curvature $c$. For such ambient spaces, Cartan \cite{isoCartan} proved that a hypersurface is isoparametric if and only if it has constant principal curvatures. In the sequence, the authors showed in \cite{reis-tenenblat} that the MCF given in this way is reduced to an ordinary differential equation, and provided explicit solutions. From such solutions, the exact collapsing times of the singularities are provided. Following the ideas of \cite{reis-tenenblat}, a version of their results was provided in \cite{freitas-abdenago}, for a class of isoparametric hypersurfaces of the product spaces $\mathbb{Q}_c^{n} \times \mathbb{R}$ and $\mathbb{Q}_c^{n} \times \mathbb{S}^1$. Recently, the author in \cite{ronaldo} also used a reparametrization of the flow by parallel hypersurfaces of isoparametric hypersurfaces to consider the Weingarten flow in Riemannian manifolds, which has as a particular case, the MCF.  In this case, following \cite{berndt}, isoparametric hypersurfaces are defined in \cite{ronaldo} as those whose parallels have constant principal curvatures, thus including the case in which the ambient spaces are space forms. For submanifolds with higher codimensions, the MCF with initial data given by an isoparametric submanifold was considered in \cite{liu-terng-1, liu-terng-2}, when the ambient space is a space form. For a class of ambient spaces (which includes the space forms with non-negative curvature), the relation between singular Riemannian foliations in which the leaves are isoparametric submanifolds (in the sense of \cite{HeintzeLiuOlmosDefIso}) with the MCF was investigated in \cite{alexandrino2020, alexandrino2016, radeschi-mcf}.

In this paper, we characterize reparametrizations of the flows by parallel hypersurfaces as the unique solution for the MCF with isoparametric hypersurfaces as initial data in general ambient spaces. Namely, for an ambient space $\widetilde{M}^{n+1}$ given by a complete Riemannian manifold such that the curvature and its covariant derivatives up to order $2$ are bounded, and with injectivity radius bounded from below by a positive constant, we prove that the MCF with bounded second fundamental form along the flow is uniquely given by a reparametrization of the flow of parallel hypersurfaces of $M^n$, if and only if $M^n$ is an isoparametric hypersurface. Our result provides an extension to general ambient spaces of \cite{reis-tenenblat,freitas-abdenago}, and an extension of \cite{ronaldo} to general isoparametric hypersurfaces, when the MCF is considered. Moreover, we also supply an improvement of their results since we show that isoparametric hypersurfaces, besides of providing solutions of the MCF through their parallel hypersurfaces, uniquely determined such evolution as initial data, as long as the ambient spaces are regular enough, which are the cases of space forms $\mathbb{Q}_c^{n+1}$, the product spaces $\mathbb{Q}_c^{n} \times \mathbb{R}$ and $\mathbb{Q}_c^{n} \times \mathbb{S}^1$, and the hyperbolic spaces $\mathbb{H}^{n+1}_{\mathbb{F}}$ considered in \cite{ronaldo}. 

A crucial element for proving our result is the use of a uniqueness theorem for the solution of MCF for general ambient spaces, which is obtained under conditions on the curvature of the ambient space, and on the second fundamental form (see \cite{ChenLeUniqueness} and Lemma 3.2 in \cite{HuiskenConvexHyper}). 

Cartan's work on isoparametric hypersurfaces of space forms \cite{isoCartan} is the main reason why isoparametric hypersurfaces are referred to as those with constant principal curvatures.
However, in arbitrary ambient spaces, the isoparametric definition and the constancy of the principal curvatures are, a priori, unrelated conditions. For instance, an isoparametric hypersurface with nonconstant principal curvatures was given in \cite{WangExIso}, and there is a recent work \cite{RVazquezExIso}, where an example of a non-isoparametric hypersurface with constant principal curvatures is provided. As an application, we use the MCF to construct a class of Riemannian manifolds that admit non-isoparametric hypersurfaces with constant principal curvatures. Such a class of ambient spaces is provided by considering a simple perturbation of the Euclidean metric (see Example \ref{ex:main2}). Therefore, this class of Riemannian manifolds provides new examples where the concept of isoparametric hypersurface is not equivalent to having constant principal curvatures. Moreover, we study the behaviour of the singularities of the MCF by a reparametrization of parallel hypersurfaces for a special class of ambient manifolds, which considers the solutions given by \cite{reis-tenenblat}. Further applications of our results are given by considering the evolution by the MCF of the isoparametric surfaces of homogeneous 3-manifolds known as $\mathbb{E}(\kappa, \tau)$, which were classified in \cite{dominguez-manzano}, and for isoparametric hypersurfaces of $\mathbb{S}^2 \times \mathbb{R}^2$ and $\mathbb{S}^2 \times \mathbb{S}^2$, which were considered in \cite{s2r2Batalla} and \cite{s2s2Urbano}, respectively. Following the ideas of \cite{s2s2Urbano}, we close the work by providing a class of isoparametric hypersurfaces in $\mathbb{H}^2 \times \mathbb{H}^2$ with three distinct principal curvatures. We also present the corresponding ODE which gives the reparametrization by parallel flow of hypersurfaces as the solution for MCF with such hypersurfaces as initial data.
			
The paper is organized as follows. In Section \ref{sec:pre_and_theo} we provide definitions, some basic concepts, and results that will be used throughout the work. In Section \ref{sec:mainResults} we state and prove our main theorems. At the end of the same section, we illustrate one of the theorems presented with the example that shows that the definitions of isoparametric hypersurfaces mentioned before are not equivalent. Section \ref{sec:examples} is devoted to the use of the results in Section \ref{sec:mainResults} to present the ordinary differential equation corresponding to the MCF when the initial data is an isoparametric surface of $\mathbb{E}(\kappa, \tau)$, as well as the mentioned classes of isoparametric hypersurfaces of product spaces.
			
\section{Background}\label{sec:pre_and_theo}

Before stating and proving our main results, let us present some background content on the mean curvature flow, isoparametric hypersurfaces, and Jacobi field theory.

Let $M^n$ be a 2-sided hypersurface of a Riemannian manifold ${M}^{n+1}$. The family of hypersurfaces $F: M^{n} \times I \rightarrow \widetilde{M}^{n+1}$, $0 \in I  \subset \mathbb{R}$, is a solution to the \textit{mean curvature flow} (abbreviated as MCF) with initial data $M$, if \begin{equation}\label{MCFdefinition}
\left\{
\begin{array}{l}
\partial_t F(x,t) = H(x,t) N(x,t),  \\
F(x,0) = x,
\end{array}
\right.
\end{equation} 
where $H(x,t)$ is the mean curvature and $N(x,t)$ is a unit normal vector field of the hypersurface $M_t:=F(M,t)$, $x \in M$.

\begin{remark}
Observe that it is enough to ask for 2-sided hypersurface, since we only need a well-defined normal vector field. In case the manifold $\widetilde{M}^{n+1}$ is orientable, such condition is equivalent to $M^n$ being an orientable hypersurface.
\end{remark}

In order to characterize the evolution by the mean curvature flow of isoparametric hypersurfaces, we will need to understand the uniqueness of the MCF. It is well-known that it holds when the initial data is compact (as can be seen, for example, in \cite{HuiskenConvexHyper}). As for the noncompact case, we will use the following result: 
\begin{theorem}[Chen-Yin \cite{ChenLeUniqueness}]\label{theo:uniq}
Let $(\widetilde{M}^{m},\widetilde{g})$ be a complete Riemannian manifold of dimension $n+1$ such that the curvature and its covariant derivatives up to order 2 are bounded and the injectivity radius is bounded from below by a positive constant, i.e., there are constants $\widetilde{C}$ and $\widetilde{\delta}$ such that $$(|\widetilde{R}| + |\widetilde{\nabla}\widetilde{R}|+|\widetilde{\nabla}^2\widetilde{R}|)(x) \leq \widetilde{C}, \,\,\, \text{inj}(\widetilde{M}^{n+1},x)>\widetilde{\delta},$$ for all $x \in \widetilde{M}^{n+1}$. Let $F_0 :M^n \rightarrow \widetilde{M}^{m}$ be an isometrically immersed Riemannian manifold with bounded second fundamental form in $\widetilde{M}^{m}$. Suppose $F_1$ and $F_2$ are two solutions to the mean curvature flow on $M^n \times [0,T]$ with the same $F_0$ as initial data and with bounded second fundamental forms on $[0,T]$; Then $F_1 = F_2$ for all $(x,t) \in M^{n} \times [0,T].$
\end{theorem}
	
In this work we will study the properties of hypersurfaces that have a particular solution for MCF:
			
\begin{definition}
Let $F:M^n \times [0,T) \rightarrow \widetilde{M}^{n+1}$ be a solution to the MCF in $\widetilde{M}^{n+1}$ with initial data $M^n$. We say that this solution is a \textit{reparametrization of the parallel flow} (abbreviated as RPF) in $[0,\delta)$, $0< \delta \leq T$, with parameter $\epsilon : [0,\delta) \rightarrow \mathbb{R}$, $\epsilon(0) = 0$ if \begin{equation}
F(x,t) = \exp_{x}(\epsilon(t) N(x)),
\end{equation} 
for all $t \in [0,\delta)$, where $\exp_{x}: T_{x}\widetilde{M} \rightarrow \widetilde{M}$ denotes the exponential map of $\widetilde{M}$ at $x \in M$, and $N$ is a unit normal vector field of the hypersurface $M$.
\end{definition}
			
As we will see in Example \ref{ex:main2}, it is possible to have a hypersurface whose solution of the MCF is initially as RPF, and after a finite time, the flow is no longer by parallels hypersurfaces.

In this work, we will use the definition of isoparametric hypersurface according to \cite{HeintzeLiuOlmosDefIso}. A hypersurface is said to be isoparametric if all nearby parallel hypersurfaces have constant mean curvature. When the ambient space is simply connected and has constant sectional curvature, this definition is equivalent to saying that the hypersurface has constant principal curvatures (see \cite{isoCartan}). As we pointed out in the introduction, this equivalence is no longer true for arbitrary ambient spaces. Following the nomenclature of \cite{DiazVazquezCastineiraDef}, a hypersurface $M$ will be called \emph{Terng-isoparametric} if the principal curvatures are constant. The ambient space given by Riemannian manifolds $\widetilde{M}$ where both definitions coincide will be called \textit{Terng-manifolds}.
			
The geometry of the parallel family of a hypersurface $M^n$ can be described in terms of the Jacobi fields of the ambient space, with initial vector tangent to the hypersurface. These fields are called $M$-Jacobi fields, and we will briefly describe them  (see \cite[Chapter 10]{bookOlmosCia} and \cite[Appendix A]{notas-miguel} for broader and more complete expositions).
			
Given a hypersurface $M^n$ of a Riemannian manifold ${M}^{n+1}$ with unit normal vector field $N$, let $r$ be a positive real number and
\begin{equation}\label{eq:ParallelFamily}
\begin{array}{rcl}
     \Phi_r: M^n & \rightarrow & \widetilde{M}^{n+1}, \\
             x &\mapsto& \exp_{x}(rN(x)), 
\end{array}
\end{equation}
where $\exp_{x}: T_{x}\widetilde{M} \rightarrow \widetilde{M}$ denotes the exponential map of $\widetilde{M}^{n+1}$ at $x \in M$. The map $\Phi_r$ is smooth and it parametrizes the parallel displacement of $M$ at a distance $r$ in the direction $N$. 
			
Let $\gamma: I \rightarrow \widetilde{M}$ be a geodesic parametrized by arc length with $0 \in I \subset \R$, $\gamma(0)=x \in M$ and $\dot{\gamma}(0) = N(x)$. Let $c: I \rightarrow M$ be a smooth curve with $c(0)=x$ and $\dot{c}(0) = X \in T_xM$. Observe that  $V(s,t) = \Phi_t(c(s)) = \gamma_s(t)$ is a smooth geodesic variation of $\gamma = \gamma_0$ with $c(s) = \gamma_s(0) \in M$ and $N \circ c(s) = \dot{\gamma}_s(0) \in (T_{c(s)}{M})^{\perp}$ for all $s$. This variation generates the Jacobi field $Y(s) = \frac{d}{ds} V(s,0)$ determined by the initial values

$$ Y(0) = \frac{d}{ds} \Big\rvert_{s=0} V(s,0) = \frac{d}{ds} \Big\rvert_{s=0} \gamma_s(0) = X \in T_xM$$ 

$$Y'(0) = \frac{D}{\partial t}\Big\rvert_{t=0} \frac{\partial}{\partial s}\Big\rvert_{s=0} V(s,t) = \frac{D}{\partial s}\Big\rvert_{s=0} \frac{\partial}{\partial t}\Big\rvert_{t=0} V(s,t) = \frac{D}{\partial s} \Big\rvert_{s=0} \dot{\gamma}_s(0) . \vspace{2mm}$$ 

Since $N \circ c(s) = \dot{\gamma}_s(0)$, using the Weingarten formula,

$$ Y'(0) = \frac{D}{\partial s} \Big\rvert_{s=0}  N = \widetilde{\nabla}_{Y(0)} N = -A_{N(x)} X. \vspace{2mm} $$ 

Thus, the initial values of $Y$ satisfy 
$$Y(0) = X(\gamma(0)) \in T_{\gamma(0)}M \text{ and}\,\, Y'(0)= -A_{N(\gamma(0))}Y(0) \in T_{x}M.$$

A Jacobi field $Y$ along $\gamma$ whose initial values satisfy these two conditions is called an $M$-Jacobi field. Roughly speaking, $M$-Jacobi fields correspond to geodesic variations of geodesics intersecting $M$ perpendicularly.

Let $D$ be the End$(\dot{\gamma}^{\bot})$-valued tensor field along $\gamma$ given as a solution of
\begin{equation}\label{jacobiEq}
D'' + \widetilde{R}(D, \dot{\gamma})\dot{\gamma} =0 ,\, D(0) = \textnormal{id}_{T_xM},\, D'(0)= - A_{N(x)},
\end{equation} where $(\cdot)'$ stands for the covariant derivative of a tensor field. If $X \in T_xM$ and $P_{\gamma}$ is the parallel transport along $\gamma$, then $Y = D\circ P_{\gamma} (X)$ is the Jacobi field along $\gamma$ with initial values $Y(0) = X$ and $Y'(0) = -A_{N(x)}X$. Let us observe that $\gamma(r)$ is a focal point of $M$ along $\gamma$ if and only if $D(r)$ is singular. If $D(r)$ is regular then the shape operator of the parallel hypersurface $M_r:=\Phi_r(M)$ associated to the unit normal $\dot{\gamma}(r)$ of the hypersurface $M_r$ is given by $A_{\dot{\gamma}(r)} = -(D'\circ D^{-1})(r).$ Consequently, by the Jacobi formula, the mean curvature fo the hypersurface $M_r$ is given by
\begin{equation}
    H(r) = - \dfrac{(\det D)'}{n \det D}(r). \label{H-parallel}
\end{equation}

We end this section with another useful formula, which is a direct consequence of the Riccati equation (see \cite[Section 3]{RVazquezExIso}):
\begin{equation}
    H'(r) = \widetilde{\textnormal{Ric}}(\dot{\gamma}(r),\dot{\gamma}(r))+ |A_{\dot{\gamma}(r)}|^2, \label{H-derivative}
\end{equation}
where $\textnormal{Ric}$ is the Ricci tensor of $\widetilde{M}$.

\section{Main Results}\label{sec:mainResults}
In this section we present our main results. We start by describing the MCF with initial data given by an isoparametric hypersurface. As application, we provide a class of Riemannian manifolds where we can find hypersurfaces with constant principal curvatures which are not isoparametric hypersurfaces. In the sequence, with initial data given by an isoparametric hypersurface, we provide sufficient conditions for singularities, when they occur, to be \textbf{Type I} singularities.

\subsection{An isoparametric hypersurface as initial data for the MCF}

Our first result is given by the Lemma below, which provides a necessary condition for a solution to the MCF to be an RPF. Lemma \ref{theo:main} has its own interest extends to Riemannian manifolds the result of \cite{reis-tenenblat} for space forms. Furthermore, the first part of Lemma \ref{theo:main} coincides with Proposition 1 \cite{ronaldo} when the MCF is considered, complementing it with the second part, {since it provides the corresponding ordinary differential equation concretely in terms of the endomorphism $D$ presented in Section \ref{sec:pre_and_theo}. For completeness, we will present its entire proof.}

\begin{lemma}\label{theo:main}
Let $M^n$ be a 2-sided hypersurface of $\widetilde{M}^{n+1}$, such that $M^n$ is the initial data of a solution $F:M \times [0,T) \rightarrow \widetilde{M}^{n+1}$ for the MCF. If $F$ restrict to $M \times [0,\delta)$ for some $0<\delta \leq T$ is a RPF with parameter $\epsilon:[0,\delta) \rightarrow \mathbb{R}$ then $M$ is an isoparametric hypersurface of $\widetilde{M}^{n+1}$. Moreover, $\epsilon$ satisfies the ODE
\begin{equation} \label{ode:MCF}
 \epsilon'(t) = - \frac{(\det D)'}{n\det D} \left( \epsilon(t) \right),
\end{equation}
where $D$ is the solution of (\ref{jacobiEq}), and the right-hand side of \eqref{ode:MCF} is independent of $x \in M$.
\end{lemma}

\begin{proof}
By hypothesis we have that $F(x,t) = \exp_{x}(\epsilon(t) N(x))$ satisfies $$\partial_tF(x,t) = H(x,t)\widetilde{N}(x,t),$$ where $\widetilde{N}(\cdot,t)$ and $H(\cdot,t)$ stand for the normal unit vector field and the mean curvute of the hypersurface $F(\cdot,t)$, for $t \in [0,\delta)$, respectively. On one hand, we have $\partial_tF(x,t)=\epsilon'(t) \left(d\exp_{x}\right)_{\epsilon(t)N(x)}N(x)$. On the other hand,  it follows from Gauss's lemma that $\left(d\exp_{x}\right)_{\epsilon(t)N(x)}N(x) = \widetilde{N}(x,t)$ for any $x \in M^{n}$ and $t \in [0,\delta)$. Thus, $\epsilon'(t) = H(x,t)$ and the hypersurface $M_t=F(M,t)$ have constant mean curvature $\epsilon'(t)$. In particular, $M^n$ is an isoparametric hypersurface.

Since $\epsilon'(t)$ is the mean curvature of the hypersurface $M_t=F(M,t)$ and the MCF is RPF with parameter $\epsilon$, we have from \eqref{H-parallel} that
$$
\epsilon'(t)=H(\epsilon(t)) = - \dfrac{(\det D)'}{n \det D}(\epsilon (t)),
$$
where $D$ is the solution of (\ref{jacobiEq}), which will be independent of the choice of $x \in M$, once $M$ is isoparametric.
\end{proof}

Lemma \ref{theo:main} says that if a solution of the MCF is given by an RPF then the initial hypersurface of this solution must be isoparametric. We will make use of it and Theorem \ref{theo:uniq} to obtain our first main result, which supplies the characterization of the MCF when the initial data is an isoparametric hypersurface, for a regular enough ambient space (in the sense of Theorem \ref{theo:uniq}).

\begin{theorem}\label{theo:converse}
Let $\widetilde{M}^{n+1}$ be a complete Riemannian manifold such that the curvature and its covariant derivatives up to order $2$ are bounded and the injectivity radius is bounded from below by a positive constant. Let $M^n$ be a hypersurface of $\widetilde{M}^{n+1}$ such that the solution $F: M^n \times [0,T) \rightarrow \widetilde{M}^{n+1}$ of the MCF with initial data $M^n$ has bounded second fundamental form on $[0,T_{-}]$ for all $T_{-}<T$. Then, $M^n$ is isoparametric if and only if $F$ is the flow by parallels for some $\delta_0 \leq T$. Moreover, suppose that  $[0,\delta)$ is the maximal interval where $F$ is a reparametrization of the parallel flow. If $\delta < T$ then $F(.,\delta)$ is a hypersurface that is not isoparametric.
\end{theorem}
			
\begin{proof}
Let $M$ be an isoparametric hypersurface of $\widetilde{M}^{n+1}$. Then the mean curvatures of its nearby parallel hypersurfaces depend only on the parallel displacement $r \geq 0$. In this case, if $D$ is a solution of \eqref{jacobiEq}, then the right-hand side of \eqref{H-parallel}, which provides the mean curvature of a parallel hypersurface of $M$, depends only on $r$. Therefore the ODE $$\epsilon'(t) = - \frac{(\det D)'}{n\det D} \left( \epsilon(t) \right), $$ 
is well defined in a neighborhood of $r=0$.
Let $\epsilon$ a solution of such ODE, with $\epsilon(0) = 0$, defined in $[0,\delta_0)$ for some $\delta_0>0$ such that $|\epsilon(t))| < \widetilde{\delta}$, for all $t \in [0,\delta_0)$ where $\widetilde{\delta}$ is the uniform bound for the injectivity radius of $\widetilde{M}.$ Thus, proceeding as in the proof of Lemma \ref{theo:main}, the family $\overline{F}: M^n \times [0,\delta) \rightarrow \widetilde{M}^{n+1}$ given by $\overline{F}(p,t) = \exp_{f(p)}(\epsilon(t)N(x))$, where $N$ is a unit normal vector field of $M^n$, whose direction is given by the vector mean curvature, is a solution of the MCF, with initial data given by $M^n$. Since the ambient space satisfies the conditions of Theorem \ref{theo:uniq} and $M^n$ is isoparametric, it follows by equation \eqref{H-derivative} that the second fundamental form is bounded for all $t \in [0,\delta)$. Consequently, it follows from Theorem \ref{theo:uniq} that $F=\overline{F}$ in $[0,\delta)$. The converse follows from Lemma \ref{theo:main}.
				
Let $[0,\delta)$ be the maximal interval where the solution of the MCF $F: M^n \times [0,T) \rightarrow \widetilde{M}^{n+1}$ is RPF. When $\delta < T$, we firstly observe that $F(.,\delta)$ is a regular hypersurface, since the $F$ is defined at $t=\delta$. Secondly, we claim that $F(M, \delta)$ is not a isoparametric hypersurface. In fact, suppose by contradiction that $F(M, \delta)$ is isoparametric. Then we can consider it as an initial data for the mean curvature flow and, by the uniqueness, we will then extend $F$ as RPF beyond $\delta$, which contradicts the maximality of $[0,\delta)$.
\end{proof}

\begin{remark}
Note that if $\delta=T$ in the previous theorem then $\delta$ will be a singularity of the MCF, whereas the first focal point of the hypersurface $M^n$ in $\widetilde{M}^{n+1}$ occurs at time $\epsilon(\delta)$. In this case, the MCF becomes extinct at the focal points of $M^n$.
\end{remark}

\begin{remark}
Theorem \ref{theo:converse} provides a refinement of Theorem 2.2 of \cite{reis-tenenblat}, when the ambient space is a space form. It assures that the unique solution for the MCF in a short time, with initial data being an isoparametric hypersurface is given by the family of parallel hypersurfaces provided by the parameter $\epsilon$ arising as the unique solution of the ordinary differential equation \eqref{ode:MCF}. Similarly, for Terng-manifolds as ambient spaces and when the MCF is considered, the unique solution by parallel hypersurfaces given in Corollary 2 of \cite{ronaldo}, will be, in fact, the unique solution to the MCF, as long as the conditions of Theorem \ref{theo:uniq} are satisfied (recall that the author in \cite{ronaldo} defines isoparametric hypersurfaces as those with constant principal curvatures).
\end{remark}

On one hand, we observe that the situation $\delta < T$ in Theorem \ref{theo:converse} does not occur if $\widetilde{M}^{n+1}$ is a Terng-manifold. This follows from the fact that by continuity, the hypersurface $F(.,\delta)$ has constant principal curvatures. On the other hand, we present below a class of examples to illustrate that what is described in item  of Theorem \ref{theo:converse} actually occurs. Such a class of examples will show that Terng-Manifolds are not the generic case of Riemannian manifolds. In this case, the solutions of the MCF will be initially given by RPF and then will change their behavior after a finite time. 

In order to produce such ambient spaces, we will make perturbations on the Euclidean metric as follows.


\begin{example}\label{ex:main2}
Let $f_{p,h,\sigma}: \mathbb{R}^{2} \rightarrow \mathbb{R}$ be a smooth bump function given by 
$$f_{p,h,\sigma}(x) = \begin{cases}
h e^{-\frac{\sigma^2}{\sigma^2 - \Vert x - p \Vert^2}} &\text{if } \Vert x - p \Vert \leq \sigma\\
0 &\text{if } \Vert x - p \Vert > \sigma
\end{cases},$$ 
$x,\, p\in \mathbb{R}^{2}$, $h,\sigma \in \mathbb{R}$, $\sigma > 0$, and $\Vert \cdot \Vert$ is the standard norm in $\mathbb{R}^{2}$. Let 
$$\widetilde{M}^{2}_{p,h,\sigma}:=\left\{(x,f_{p,h,\sigma}(x)),\, x \in \mathbb{R}^2 \right\} \subset \mathbb{R}^3$$ be the Riemannian manifold given by the graph of $f$, endowed with the induced metric. By taking the standard Riemannian product $\widetilde{M}^{n+1}:=\widetilde{M}^2_{p,h,\sigma} \times \mathbb{R}^{n-1}$, it is straightforward to see that $\widetilde{M}^{n+1}$ satisfies the conditions of Theorem \ref{theo:uniq}.

Now we consider the curve $M_{p,h,\sigma}^1 \subset \widetilde{M}^{2}_{p,h,\sigma}$, given by
$$M_{p,h,\sigma}^1= \{(x,f_{p,h,\sigma}(x))  \in \widetilde{M}^{2}_{p,h,\sigma} : \Vert x- \mathcal{O} \Vert = R\},$$ for some $\mathcal{O} \in \R^{2}$ and $R>\sigma > 0$ such that $\Vert p - \mathcal{O} \Vert <  R - \sigma$. Since $\Vert x- \mathcal{O} \Vert = R$, we conclude that $\Vert x- p \Vert > \sigma$ and then $f_{p,h,\sigma}(x)=0$. From $M^1_{p,h,\sigma}$, we construct the $(n-1)$-cylinder $M^n:=M^1_{p,h,\sigma} \times \mathbb{R}^{n-1} \subset \widetilde{M}^{n+1}$ (Figure \ref{step-1}). 

\begin{figure}[htbp]
\includegraphics[scale=1]{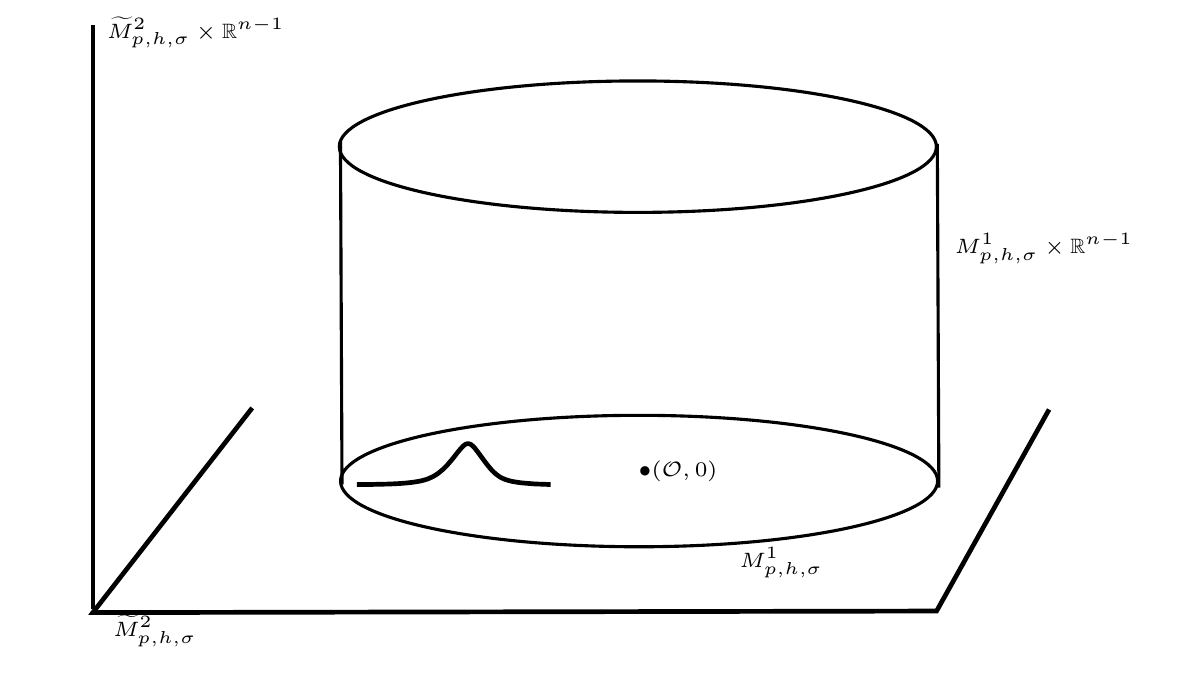}
\caption{Ambient space and initial data}
\label{step-1}
\end{figure}

At this point, it is important to emphasize two facts:
\begin{enumerate}[a)]
    \item Endowing $\mathbb{R}^3$ with canonical coordinates $(x_1,\,x_2,\,x_3)$, we conclude that $M_{p,h,\sigma}^1$ is an Euclidean circle in the plane $[x_3=0]$ with radius $R$ and center $(\mathcal{O},0)$, lying in the flat part of $\widetilde{M}^{2}_{p,h,\sigma}.$ However, $M_{p,h,\sigma}^1$, when viewed as a curve in $\widetilde{M}^{2}_{p,h,\sigma}$,  is not a geodesic circle centered in $(\mathcal{O},0)$, due to the presence of the bump inside the region bounded by $M_{p,h,\sigma}^1$;
    \item Since $M_{p,h,\sigma}^1 \times \mathbb{R}^{n-1}$ is entirely contained in $\mathbb{R}^{n+1} \setminus \mathcal{C}$, where $\mathcal{C}$ is the solid cylinder $\overline{B_{\sigma}(p)} \times \mathbb{R}$, and $M_{p,h,\sigma}^1 \times \mathbb{R}^{n-1}$ is a cylinder over $\mathbb{S}^1(R)$, we conclude that $M_{p,h,\sigma}^1 \times \mathbb{R}^{n-1}$ is an isoparametric hypersurface in $\widetilde{M}^{n+1}$.
\end{enumerate}
We are in position to apply Theorem \ref{theo:converse}. Since $M^n$ is an isoparametric hypersurface, we have that a solution $F:M^n \times [0,T) \rightarrow \widetilde{M}^{n+1}$ for the MCF with initial condition given by $M^n$ is RPF in $[0,\delta)$ with $0<\delta \leq T$. Furthermore, the fact that $\mathbb{R}^{n-1}$ is totally geodesic in $\widetilde{M}^{n+1}$, added to the uniqueness of the MCF, implies that the solution $F$ is given by $F((q,v),t) = (\gamma_t(q),v) $, where $\gamma_t$ is a solution to the shortening flow in $\widetilde{M}^2_{p,h,\sigma}$, evolving by parallel curves, with initial data given by $M^1_{p,h,\sigma}$.

By the main theorem in \cite{graysonShort}, the curve shortening flow in a surface $M^2$ has a solution for time $t \in [0, T)$, denoted by $\gamma_t$, where $T \in [0,\infty) \cup \{\infty\}$ and the following hold:
\begin{enumerate}[(i)]
\item  If $T$ is finite, $\gamma_t$ shrinks to a point;
\item  If $T$ is infinite, $\gamma_t$ approaches a geodesic.
\end{enumerate}

We claim that the MCF with inital data $M_{p,h,\sigma}^1 \times \mathbb{R}^{n-1}$ is no longer by parallels after a finite time. In fact, suppose by contradiction that $\delta = T$, consider a half-space $\Omega \subset \widetilde{M}^{2}_{p,h,\sigma}$ such that $(\mathcal{O},\,0) \in \partial \Omega$ and $\Omega$ is flat. Observe that $M_{p,h,\sigma}^1 \cap \Omega$ has $(\mathcal{O},\,0) \in \widetilde{M}^{2}_{p,h,\sigma}$ as the focal point in the induced metric, and $\gamma_t \cap \Omega$ is a semicircle centered on the same point for every $t\in[0,T)$. Thus, $T$ is finite, otherwise the curvature of $\gamma_t$ must go to zero, which is not the case, as $\gamma_t \cap \Omega$ is a semicircle in a flat space for all $t\in [0,T)$. Since we suppose that the MCF is RPF, $\gamma_t$ shrinks to the point $(\mathcal{O},\,0) \in \widetilde{M}^{2}_{p,h,\sigma}$. On one hand, at some time $t_0 \in [0,T)$ the  curve $\gamma_{t_0}$ is sufficiently close to $(\mathcal{O},\,0) \in \widetilde{M}^{2}_{p,h,\sigma}$ such that is entirely contained in an open subset of the plane $[x_3=0]$. Since the MCF is RPF, $\gamma_{t_0} \times \mathbb{R}^{n-1}$ is an isoparametric hypersurface of a subset of the Euclidean space. It follows from the classification of isoparametric hypersurfaces of $\mathbb{R}^{n+1}$ (\cite{levi, segre, somigliana}, see also \cite{notas-miguel} for a unified proof), and from the geometry of $\gamma_{t_0} \cap \Omega$ that $\gamma_{t_0} \times \mathbb{R}^{n-1}$ must be a cylinder over a circle centered at $\{ (\mathcal{O},\,0) \}$ (Figure \ref{step-2}). On the other hand, due to the metric perturbation by the bump function, as pointed out in the item a) above, and the fact that the MCF is RPF, not all points of the circle $\gamma_{t_0}$ are equidistant from the point $\{ (\mathcal{O}, \, 0) \}$ which generates a contradiction.

\begin{figure}[htbp]
\includegraphics[scale=1]{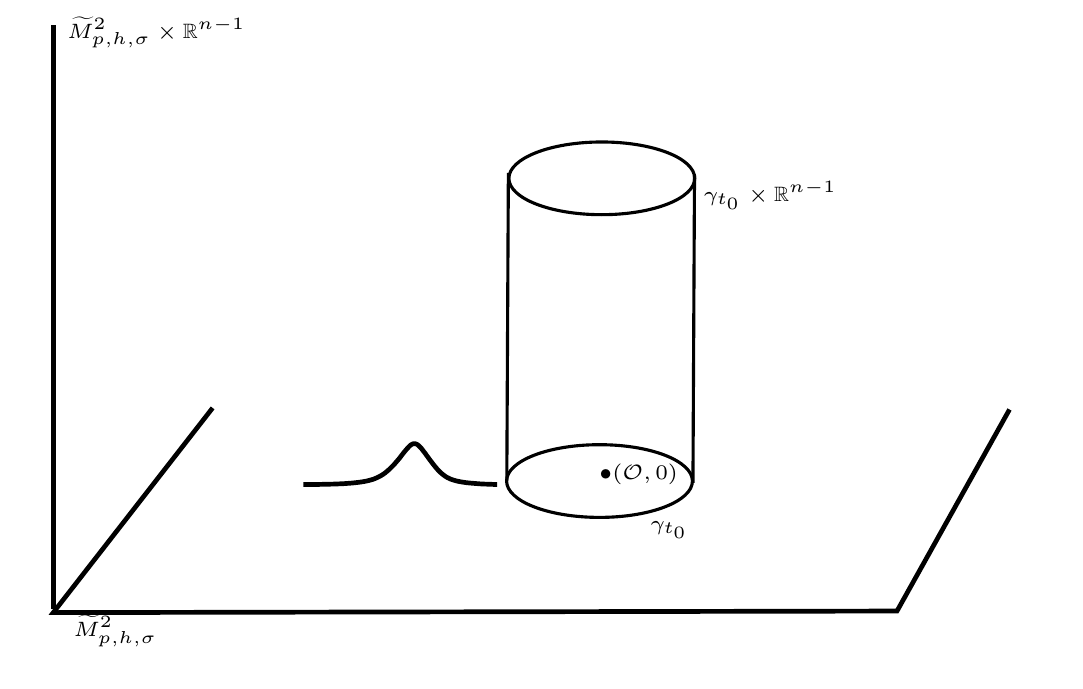}
\caption{Evolution at $t=t_0$}
\label{step-2}
\end{figure}

We conclude that $M_t=(\gamma_t(q),v)$ are isoparametric hypersurfaces for $t<\delta$,but are not for $t=\delta$. Otherwise, by Theorem \ref{theo:converse}, we could continue the flow evolving by parallels. Furthermore, $M_t$ has only one nonzero principal curvature, which depends only on $t$. By continuity, the principal curvatures of $M_{\delta}$ are constant. Consequently, $\widetilde{M}^{2}_{p,h,\sigma}\times \mathbb{R}^{n-1}$ is not a Terng-manifold.
\end{example}

			
\subsection{About the behaviour of the singularities}\label{sec:sing}

In this subsection we study the singularities which arises when we evolve isoparametric hypersurfaces $f:M \rightarrow \widetilde{M}^{n+1}$ by the mean curvature flow, for regular enough ambient spaces. Let us recall that, if $F: M^n \times I \rightarrow \widetilde{M}^{n+1}$ is a solution to the MCF in which $M_t := F(M,t)$ becomes singular at $t=T<\infty$, then we have a formation of a singularity of the flow. Following \cite{Isenberg2019, Isenberg2020, surveySmoczyk}, we say that such a singularity is \emph{Type I} if $$ \lim_{t \to T} (T-t) \displaystyle \sup_{M_t} |A|^2 < \infty.$$ Otherwise, we say that the singularity is \emph{Type II}.
			
In order to study the singularities, we will need one of the evolution equations given in \cite{HuiskenConvexHyper}. Namely, given a  solution $F: M \times [0,T) \rightarrow \widetilde{M}^{n+1}$ to the MCF, let $\{e_i\}$ be a orthonormal basis of $\widetilde{M}^{n+1}$ at $F(x_0,t_0)$ such that $e_0 = N(x_0,t_0)$ and $\{e_i\}_{1\leq i \leq n}$ is a geodesic frame of $M_{t_0}$ at $(x_0,t_0)$, thus
			
\begin{equation}\label{EvoEq}
\begin{aligned}
\frac{\partial}{\partial t} |A|^2 &= \Delta |A|^2 - 2 |\nabla A|^2 + 2 |A|^2 (|A|^2 + \text{Ric}(N,N)) \\
&- 4(h^{ij}{h_j}^m{\widetilde{R}_{mli}}{}^l - h^{ij}h^{lm}\widetilde{R}_{milj}) - 2h^{ij}(\widetilde{\nabla}_j\widetilde{R}_{0li}{}^l + \widetilde{\nabla}_l{\widetilde{R}_{0ij}}{}^l).
\end{aligned}
\end{equation} 
where $h_{ij} = \langle A_t e_i, e_j\rangle$.
			
Next, we show that, at Terng-manifolds, which satisfy the conditions of Theorem \ref{theo:uniq}, the singularities of the MCF, when they occur, are Type I singularities, if the initial data is an isoparametric hypersurface.

\begin{theorem}\label{prop:Sing}
Let $M^n$ be an isoparametric hypersurface of a Terng-manifold $\widetilde{M}^{n+1}$, which satisfies the conditions of Theorem \ref{theo:uniq}. Then, if there is a singularity at a finite time $T\in \mathbb{R}$ for the MCF with initial data given by $M^n$, then the singularity must be of Type I.
\end{theorem}
			
\begin{proof} Since $M^n$ be an isoparametric hypersurface of a Terng-manifold $\widetilde{M}^{n+1}$, the shape operator along the flow depends only on $t$. Then, the evolution equation (\ref{EvoEq}) is given by
\begin{equation*}
\begin{split}
\frac{d}{dt} |A|^2 = 2 |A|^2 (|A|^2 + \text{Ric}_{\widetilde{M}}(N,N))
- 4(h^{ij}{h_j}^m{\widetilde{R}_{mli}}{}^l - h^{ij}h^{lm}\widetilde{R}_{milj}) \\ - 2h^{ij}(\widetilde{\nabla}_j\widetilde{R}_{oli}{}^l + \widetilde{\nabla}_l{\widetilde{R}_{oij}}{}^l).
\end{split}
\end{equation*}
Under the curvature conditions of Theorem \ref{theo:uniq} on the ambient space, we have
\begin{equation*}
\begin{split}
\frac{d}{dt} |A|^2 \geq 2 |A|^2 (|A|^2 - \text{sup}_{\widetilde{M}}|\text{Ric}_{\widetilde{M}}|)
- 8|A|^2(\text{sup}_{\widetilde{M}}|\widetilde{R}|) - 4|A|(\text{sup}_{\widetilde{M}}|\widetilde{\nabla}\widetilde{R}|).
\end{split}
\end{equation*}
			
If $|A|^2 \geq |A|$, we can go further and write
\begin{equation}
\begin{aligned}
\frac{d}{dt} |A|^2 &\geq 2 |A|^2 (|A|^2 - \text{sup}_{\widetilde{M}}|\text{Ric}_{\widetilde{M}}|-4\text{sup}_{\widetilde{M}}|\widetilde{R}| - 2\text{sup}_{\widetilde{M}}|\widetilde{\nabla}\widetilde{R}|) \\
&=2 |A|^2 (|A|^2 - \widetilde{C}),
\end{aligned} \label{estimate-A}
\end{equation}
where $\widetilde{C}=\text{sup}_{\widetilde{M}}|\text{Ric}_{\widetilde{M}}|+4\text{sup}_{\widetilde{M}}|\widetilde{R}| + 2\text{sup}_{\widetilde{M}}|\widetilde{\nabla}\widetilde{R}|$.

Once there is a singularity at $T$, $|A|^2$ goes to infinity as $t$ goes to $T$. Therefore, we can take $t_0$ close enough to $T$ and assure $|A|^2 \geq |A|$ and $|A|^2 \geq \widetilde{C}$, for $t>t_0$. In this case, it follows from \eqref{estimate-A} that

$$\ln\left(\frac{|A|^2(t)}{|A|^2(t) - \widetilde{C}}\right)\geq 2(T-t) + \lim_{s \rightarrow T^{-}}\ln\left(\frac{|A|^2(s)}{|A|^2(s) -\widetilde{C}}\right),$$ 
for $t>t_0$. Since $\displaystyle \lim_{s \rightarrow T^{-}}|A|^2(s) = \infty$, the following estimate holds for $|A|^2:$
$$|A|^2(t) \leq \dfrac{\widetilde{C}}{e^{2(T-t)}-1}.$$
Consequently, $$\lim_{t \mapsto T} (T-t) \text{sup}_{M_t} |A|^2 \leq \lim_{t \mapsto T}\frac{(T-t)\widetilde{C}}{e^{2(T-t)}-1} = \frac{\widetilde{C}}{2nc}<  \infty.$$
				
\end{proof}

\begin{remark}
i) Under the conditions of Theorem \ref{theo:converse}, it follows that the singularities for the MCF with isoparametric hypersurfaces with initial data, if they occur, will appear when the parameter $\epsilon$ for the reparametrization by the parallel flow become singular. If, in addition, the ambient space is a Terng-manifold, Theorem \ref{prop:Sing} asserts that such singular point for $\epsilon$ will provide a Type I singularity. \vspace{0.2cm} \\ 
\noindent ii) By solving the ordinary differential equations \eqref{ode:MCF} corresponding to isoparametric of space forms, the authors in \cite{reis-tenenblat} provided explicitly the singular points $\epsilon$. Since the space forms are Terng-manifolds \cite{isoCartan} and satisfy the conditions of Theorem \ref{theo:converse}, we conclude that the singularities of the MCF in space forms, with isoparametric hypersurfaces as initial data, are all of Type I.
\end{remark}

\section{Applications}\label{sec:examples}

In this section, we apply the results of Section \ref{sec:mainResults} to study the evolution of isoparametric hypersurfaces of important ambient space in the literature: the homogeneous 3-manifolds $\mathbb{E}(\kappa, \tau)$, and the Riemannian products $\mathbb{S}^2 \times \mathbb{R}^2$, $\mathbb{S}^2 \times \mathbb{S}^2$, and $\mathbb{H}^2 \times \mathbb{H}^2$. Furthermore, we provide a class of isoparametric hypersurfaces of $\mathbb{H}^2 \times \mathbb{H}^2$, in an analougous way it was done in \cite{s2s2Urbano}. We highlight that all ambient spaces considered are homogeneous manifolds. Therefore, they satisfy the conditions of Theorem \ref{theo:uniq}.

\subsection{On surfaces in $\mathbb{E}(\kappa, \tau)$} \label{subsec:example_KapaTau} In this subsection we consider isoparametric surfaces of the homogeneous Riemannian $3$-manifolds $\mathbb{E}(\kappa, \tau)$, with $4$-dimensional isometry group. Following \cite{daniel, dominguez-manzano}, given $\kappa, \tau \in \mathbb{R}$, with $\kappa - 4 \tau^2 \neq 0$, such manifolds are given as total spaces of a Riemannian submersion $\pi : \mathbb{E}(\kappa,\tau) \rightarrow \mathbb{Q}^{2}_{\kappa}$ whose fibers are the integral curves of a unitary Killing vector field $\xi$, with bundle curvature $\tau$. When $\tau=0$, we have the standard Riemannian products $\mathbb{Q}^{2}_{\kappa} \times \mathbb{R}$. The case $\tau \neq 0$ provides us the Berger spheres ($\kappa > 0$), the Heisenberg space Nil$_3$ ($\kappa = 0$) and the universal cover of the Lie group PSL$_2(\mathbb{R})$.

 In \cite{dominguez-manzano} it was shown that $\mathbb{E}(\kappa,\tau)$ is a Terng-manifold. More precisely, they gave a classification of the isoparametric hypersurfaces of such space, and using this result we will solve the ODE \eqref{ode:MCF} for the non-trivial cases.
			
			
\begin{theorem}[Dom\'inguez-V\'azquez and Manzano \cite{dominguez-manzano}]\label{theo:miguel_manzano}
Let $\Sigma$ be an immersed isoparametric sufrace in $\mathbb{E}(\kappa,\tau)$, $\kappa-4\tau^2 \neq 0$. Thus, $\Sigma$ is an open subset of one of the following complete surfaces:
\begin{enumerate}[(a)]
\item a vertical cylinder over a complete curve of constant curvature in $\mathbb{Q}^2_\kappa$, \label{item:vertical_cyl}
\item a horizontal slice in $\mathbb{Q}^2_\kappa \times \left\{t_0\right\}$ with $\tau =0$,\label{item:slice}
\item a parabolic helicoid $P_{H,\kappa,\tau}$ with $4H^2 + \kappa < 0$. \label{item:parabolic_helicoid}
\end{enumerate}
\end{theorem}
		
Assume that $4H^2 + \kappa < 0$, and consider the halfspace model $\mathbb{E}(\kappa,\tau) = \{(x,y,z) \in \mathbb{R}^3 : y>0\}$ endowed with the Riemannian metric $$\frac{dx^2+dy^2}{-\kappa y^2}+\left(dz - \frac{2\tau}{\kappa y}dx\right)^2.$$ In this model, $P_{H,\kappa,\tau}$ is the entire $H$-graph parameterized by $$X(u,v) = (u,v,a\log(v)), \,\text{ with}\,\, a= \frac{2H\sqrt{-\kappa +4\tau^2}}{-\kappa\sqrt{-4H^2-\kappa}}.$$ We hightlight that $P_{H,\kappa,\tau}$ has constant angle function $\nu = \langle \xi, N \rangle$ and satisfies $\nu^2 = \dfrac{4H^2+\kappa}{\kappa - 4\tau^2}$, where $N$ is a unit normal vector field.


We will study the MCF with initial data given by the surfaces described in item \eqref{item:vertical_cyl} or \eqref{item:parabolic_helicoid} in Theorem \ref{theo:miguel_manzano}. Since the horizontal slices given in item \eqref{item:slice} are totally geodesic, the MCF in this case is trivial. Following the notation in \cite{dominguez-manzano}, all the calculations will be done in the orthogonal basis $\left\{ {U}/{|U|}, \, {JU}/{|JU|} \right\},$ where $U = \xi - \nu N$ is the tangent part of the unit Killing vector field $\xi$, and $J$ is the $\frac{\pi}{2}$-rotation, that is, $Jv = v \wedge N$ for $v \in TM$.
			
Let $\Sigma$ be an immersed surface in $\mathbb{E}(\kappa,\tau)$, it follows from Proposition 3.3 in \cite{daniel} that the shape operator $A_0$ of $\Sigma$ is given by
\begin{equation}
A_0 = \left(
\begin{array}{cc}
- \dfrac{d \nu(U)}{|U|^2}   &  - \dfrac{d \nu(JU)}{|U|^2} - \tau  \\
- \dfrac{d \nu(JU)}{|U|^2} - \tau   & \lambda
\end{array}
\right) \label{shape-operator-kt}
\end{equation}
where $\lambda = \dfrac{\langle A_0 JU, JU \rangle}{|U^2|}$.

As given in \cite{dominguez-manzano}, the isoparametric surfaces in $\mathbb{E}(\kappa, \tau)$ have constant angle function $\nu$. Therefore, in this case, we rewrite the matrix \eqref{shape-operator-kt} as
\begin{equation}
A_0 = \left(
\begin{array}{cc}
0  &   - \tau  \\
- \tau   & 2 H
\end{array}
\right) \label{shape-operator-kt-isoparametric}
\end{equation}\\
			
In order to obtain the ODE (\ref{ode:MCF}) we will need the expression of operator $D(r)$ associated to the family of surfaces $\Phi_r(x) = \exp_{x}(rN(x))$. From \cite{dominguez-manzano}, and using the shape operator $A_0$ of the initial surface, we have that:
			
\begin{equation}
D(r) = \left(
\begin{array}{cc}
1  &   2\tau s_{\delta}(r) - 4\delta^{-1}\tau H(c_\delta(r)-1)  \\
0   & - 2Hs_\delta(r)+c_\delta(r)
\end{array}
\right),
\end{equation}\\
where $\delta = (\kappa - 4\tau^2)\nu^2 - \kappa$, $\nu$ is the angle function, and we have considered the auxiliary functions
\begin{equation*}
s_\delta(t)=\begin{cases}\frac{1}{\sqrt{\delta}}\sinh(t\sqrt{\delta})&\text{if }\delta>0,\\
\frac{1}{\sqrt{-\delta}}\sin(t\sqrt{-\delta})&\text{if }\delta<0,
\end{cases}\qquad 
c_\delta(t)=\begin{cases}\cosh(t\sqrt{\delta})&\text{if }\delta>0,\\
\cos(t\sqrt{-\delta})&\text{if }\delta<0.
\end{cases}
\end{equation*}

We are in a position to present the ODE \eqref{ode:MCF} of the case \textit{(a)} in Theorem \ref{theo:miguel_manzano}. Let $\Sigma$ be a vertical cylinder  over a complete curve of constant curvature $\kappa_g=2H$ in $\mathbb{Q}^2_{\kappa}$. In this case we have $\nu \equiv 0$, and the constant $\delta$ is given by $-\kappa$. It follows that $\det(D(r)) = - k_g s_{-\kappa} (r) + c_{-\kappa} (r)$, and the ODE \eqref{ode:MCF} is given by 
$$2 \epsilon'(t) - \dfrac{k_gc_{-\kappa}(\epsilon(t))  + \kappa s_{-\kappa} (\epsilon(t))}{-k_g s_{-\kappa} (\epsilon(t)) + c_{-\kappa} (\epsilon(t))} = 0,$$
which is reduced to 
$$\dfrac{2}{\kappa} \dfrac{d}{dt} \left[ k_gc_{-\kappa}(\epsilon(t))  + \kappa s_{-\kappa} (\epsilon(t)) \right] = k_gc_{-\kappa}(\epsilon(t))  + \kappa s_{-\kappa} (\epsilon(t)).$$
Thus, the general solution has the form
$$c_{-\kappa}(\epsilon(t))  + \frac{\kappa}{k_g} s_{-\kappa} (\epsilon(t)) = e^{\dfrac{kt}{2}}.$$

It remains to study the case \textit{(c)} in Theorem \ref{theo:miguel_manzano}. Let $\Sigma$ be the parabolic helicoids $P_{H,\kappa,\tau}$. As before, since $\nu^2 = \dfrac{4 H^2 + \kappa}{\kappa - 4 \tau^2}$, the constant $\delta$ is given by $4 H^2$ and, consequently, $\det D(r) = - 2H s_{4 H^2} (r) + c_{4 H^2} (r)$. Since $s_\delta' = c_\delta$ and $c_\delta'= \delta s_\delta$ the ODE \eqref{ode:MCF} is given by
$$2 \epsilon'(t) + \dfrac{-2Hc_{4H^2}(\epsilon(t)) + {4H^2} s_{4H^2} (\epsilon(t))}{-2 H s_{4H^2} (\epsilon(t)) + c_{4H^2} (\epsilon(t))} = 0.$$ Thus $$2 \epsilon'(t) - 2H = 0,$$ 
and $\epsilon(t) = Ht$, which is defined for all $t \in \mathbb{R}$.

\begin{remark}
 The classification given in \cite{dominguez-manzano} tells us that the $\mathbb{E}(\kappa,\tau)$ are Terng-manifolds. As we observed in Section \ref{sec:mainResults}, it follows from Theorem \ref{theo:converse} that the singular time of the MCF is given by the singular time of the RPF. Following exact the same calculation of Proposition 2.6 and Proposition 2.8 in \cite{reis-tenenblat}, we have that \begin{itemize}
    \item If $\kappa<0$ and $0<|k_g|<1$, the solution is defined for all $t\in \mathbb{R}$;
    \item If $\kappa<0$ and $|k_g|>1$, the singularity ocurrs at time $T = \frac{1}{4}\text{ln}\left(\frac{(k_g/\kappa)^2}{(k_g/\kappa)^2-1}\right)$;
    \item If $\kappa>0$, the singularity ocurrs at time $T = \frac{1}{4}\text{ln}\left(\frac{(k_g/\kappa)^2+1}{(k_g/\kappa)^2}\right)$.
\end{itemize}
\end{remark}


\subsection{On hypersurfaces of $\mathbb{S}^2\times \mathbb{R}^2$}\label{subsec:SR}
Isoparametric hypersurfaces with constant principal curvatures in $\mathbb{S}^2\times \mathbb{R}^2$ are of the form $\mathbb{S}^2\times \mathbb{R}$, $\mathbb{S}^2\times \mathbb{S}^1(b)$ (for $b \in \mathbb{R}^+$) or $\mathbb{S}^1(a)\times \mathbb{R}^2$ (for $a \in (0,1)$) where $\mathbb{S}^1(r)$ is the circle with radius $r \in \mathbb{R}$, see \cite{s2r2Batalla}. These hypersurfaces are characterized by a constant function $C$ defined as $C = \langle PN,N \rangle$, where $P$ is a product structure in $\mathbb{S}^2\times \mathbb{R}^2$ defined by $P(v_{1},v_{2}) = (v_{1},-v_{2})$ and $N$ is the unit normal. On these hypersurfaces, the function $C$ assumes the values $1$ or $-1$.

Let us look at each case separately. 

First, for $\mathbb{S}^2\times \mathbb{R}$, we have $C=-1$ and the unit normal is of the form $N = (0,N_{2})$, where $N_{2}$ is the component of $N$ in $\mathbb{R}^2$ with $\vert N_{2} \vert^2 = \frac{1 - C}{2}$. Given $v = (v_{1},v_{2}) \in T(\mathbb{S}^2\times \mathbb{R})$, we have
$$
S_{N}(v) = - \Tilde{\nabla}_{v}N = -\nabla^{\mathbb{R}^{2}}_{v_2}N_{2} = -dN_{2}(v_{2}) = 0,
$$
where $\Tilde{\nabla}$ is the Levi Civita connection of $\mathbb{S}^2\times \mathbb{R}^2$. Then, we have that  $\mathbb{S}^2\times \mathbb{R}$ is totally geodesic and $H=0$. Since we are in the conditions of Theorem \ref{theo:uniq}, the flow is stationary.

For $\mathbb{S}^2\times \mathbb{S}^{1}(b)$ (for $b \in \mathbb{R}^+$), we also have $C=-1$ and $N = (0,N_{2})$. Given $w = (w_{1},w_{2}) \in T(\mathbb{S}^2\times \mathbb{S}^{1}(b))$, we get
$$
S_{N}(w) = - \Tilde{\nabla}_{w}N = -\nabla^{\mathbb{R}^{2}}_{w_2}N_{2} = -dN_{2}(w_{2}) = -\frac{1}{b}w_{2}.
$$
Then, given an orthonormal basis $\{u_{1}, u_{2}, u_{3} \}$ in $\mathbb{S}^2\times \mathbb{S}^{1}(b)$, with $u_{1}, u_{2} \in T\mathbb{S}^2$ and $u_{3} \in T\mathbb{S}^{1}(b)$, we have
\begin{equation*}
S_{N}(u_{1}) = S_{N}(u_{2}) = 0 \quad \mbox{and} \quad S_{N}(u_{3}) = -\frac{1}{b}u_{3},
\end{equation*}
that is, $H = -\dfrac{1}{3 b}$. Moreover, the displacement of $\mathbb{S}^2\times \mathbb{S}^{1}(b)$ in direction $N$ at distance $r$ is given by 
\begin{align}\label{eq:disp-S2xS1}
    \begin{split}
    \Phi_r(p,q) &= \exp_{(p,q)}\left(r N_{(p,q)}\right) \\
                &= \bigg(p, q + rN_{2}(q)\bigg) \\
                &= \mathbb{S}^2 \times \mathbb{S}^1(b+r)
    \end{split}
\end{align}

The MCF with initial data $M_s$ is given by $\Phi_{\epsilon(t)}$, where $\epsilon$ is the solution of the ODE \eqref{ode:MCF}:
\begin{equation}
	\epsilon'(t) =  -\dfrac{1}{3(\epsilon (t)+b)},
\end{equation}
that is, $3(\epsilon(t)+b)^2 = K_{1}-2t$, where $K_{1}$ is a constant. Therefore, $\epsilon(t)=\sqrt{b^2-\dfrac{2t}{3}}-b$.

Finally, for $\mathbb{S}^1(a)\times \mathbb{R}^2$ (for $a \in (0,1)$), we have $C = 1$ and the unit normal is of the form $N = (N_{1},0)$, where $N_{1}$ is the component of $N$ in $\mathbb{S}^2$ with $\vert N_{1} \vert^2 = \frac{1 + C}{2}$. Given $u = (u_{1},u_{2}) \in T(\mathbb{S}^1(a)\times \mathbb{R}^2)$, we have
$$
S_{N}(u) = - \Tilde{\nabla}_{u}N = -\nabla^{\mathbb{S}^{2}}_{u_{1}}N_{1} = -dN_{1}(u_{1}) = \cot(\phi_a)u_{1},
$$
where $0 < \phi_a < \pi$ and $\Tilde{\nabla}$ is the Levi Civita connection of $\mathbb{S}^2\times \mathbb{R}^2$. Then, given an orthonormal basis $\{v_{1}, v_{2}, v_{3} \}$ in $\mathbb{S}^1(a)\times \mathbb{R}^2$, with $v_{1} \in T\mathbb{S}^{1}(a)$ and $v_{2}, v_{3} \in T\mathbb{R}^2$, we have
\begin{equation*}
S_{N}(v_{1}) = \cot(\phi_a) v_{1} \quad	\mbox{and} \quad S_{N}(v_{2}) = S_{N}(v_{3}) = 0,
\end{equation*}
that is, $H_{\phi} = \dfrac{\cot(\phi_a)}{3}$. The displacement of $\mathbb{S}^1(a)\times \mathbb{R}^2$ in direction $N$ at distance $r$ is given by

\begin{align*}
	\begin{split}
		\Phi_r(p,q) = \bigg((\cos{r})p + (\sin{r})N_{1}(p), q \bigg).
	\end{split}
\end{align*}

The MCF with initial data $M_s$ is given by $\Phi_{\epsilon(t)}$, where $\epsilon$ is the solution of the ODE \eqref{ode:MCF}:
\begin{equation}
	\epsilon'(t) = \dfrac{\cot(\phi_a-\epsilon(t))}{3},
\end{equation}
that is, $\cos(\phi_a-\epsilon(t))=\cos(\phi_a) e^\frac{t}{3}$.

\subsection{On hypersurfaces of $\mathbb{S}^2\times \mathbb{S}^2$}
In this case, the isoparametric hypersurfaces were classified in \cite{s2s2Urbano}. They are congruent to $\mathbb{S}(a)\times \mathbb{S}^2$, $a \in (0,1]$, or to $M_t$, $t\in (-1,1)$, which is defined as 
$$ M_t = \{(p,q) \in \mathbb{S}^2\times \mathbb{S}^2 \hookrightarrow \mathbb{R}^3 \times \mathbb{R}^3: \langle p,q \rangle_{\mathbb{R}^3} = t\}.$$
The solution of the MCF with initial data $\mathbb{S}(a)\times \mathbb{S}^2$ is essentially the same as in subsection \ref{subsec:SR}, so here we will present the ODE \ref{ode:MCF} of the MCF with initial data $M_t$ for $t\in (-1,1)$.

In this subsection the products are all in $\mathbb{R}^3$, in \cite{s2s2Urbano} it was provided the normal vector field $$N_{(p,q)} = \frac{1}{\sqrt{1-t^2}} \left(q-tp, p-tq \right),$$ and the mean curvature $H_t = \frac{\sqrt{2}t}{3\sqrt{1-t^2}}$ of $M_t$. Observe that the displacement of $M_t$ in direction $N$ at distance $r$ is given by

\begin{align*}
    \begin{split}
    \Phi_r(p,q) &= \exp_{(p,q)}\left(r N_{(p,q)}\right) \\
                &= \left(\left(\cos\frac{r}{\sqrt{2}}\right)p + \left(\sin\frac{r}{\sqrt{2}}\right)\frac{q-tp}{\sqrt{1-t^2}}, \left(\cos\frac{r}{\sqrt{2}}\right)q + \left(\sin\frac{r}{\sqrt{2}}\right)\frac{p-tq}{\sqrt{1-t^2}}\right)\\
                &= \left(\mathcal{P}_r(p,q), \mathcal{Q}_r(p,q)\right).
    \end{split}
\end{align*}
\noindent Since $\left\langle\mathcal{P}_r(p,q), \mathcal{Q}_r(p,q) \right\rangle = t\cos\left( \sqrt{2}r\right) + \sqrt{1-t^2}\sin\left( \sqrt{2}r\right)$, it follows that $\Phi_r (M_t) = M_{\phi(r,t)}$, where $\phi(r,t) = t\cos\left( \sqrt{2}r\right) + \sqrt{1-t^2}\sin\left( \sqrt{2}r\right)$.

The MCF with initial data $M_s$ is given by $\Phi_{\epsilon(t)}$, where $\epsilon$ is the solution of the ODE \eqref{ode:MCF}:
\begin{equation}
    \begin{split}
        \epsilon'(t) &= H_{\phi(\epsilon(t),s)}\\
                     &= \frac{\sqrt{2}\phi(\epsilon(t),s)}{3\sqrt{1-\phi(\epsilon(t),s)^2}}\\
                     &= \frac{\sqrt{2}\left(s\cos\left( \sqrt{2}\epsilon(t)\right) + \sqrt{1-s^2}\sin\left( \sqrt{2}\epsilon(t)\right) \right)}{3\sqrt{1-\left( s\cos\left( \sqrt{2}\epsilon(t)\right) + \sqrt{1-s^2}\sin\left( \sqrt{2}\epsilon(t)\right) \right)^2}}.
    \end{split}
\end{equation}

\subsection{On hypersurfaces of $\mathbb{H}^2\times \mathbb{H}^2$} Unlike the space $\mathbb{S}^2\times \mathbb{S}^2$, there is no classification of isoparametric hypersurfaces in $\mathbb{H}^2\times \mathbb{H}^2$. However, following the ideas of \cite{s2s2Urbano}, we build examples of isoparametric hypersurfaces with three distinct (constant) principal curvatures, and provide the ODE \eqref{ode:MCF} of the MCF whose such hypersurface is the initial data.

Let $$ \widetilde{M}_t = \{(p,q) \in \mathbb{H}^2\times \mathbb{H}^2 \hookrightarrow \mathbb{L}^3 \times \mathbb{L}^3: \langle p,q \rangle_{\mathbb{L}^3} = -t\},$$ for $t>1$. In this subsection all products will be taken in $\mathbb{L}$. Then it is easy to check that $\widetilde{M}_t$ is a hypersurface of $\mathbb{H}^2\times \mathbb{H}^2$ with normal vector field $$N_{(p,q)} = \frac{1}{\sqrt{2(-1+t^2)}} \left(q-tp, p-tq \right).$$ Let $(v_1,v_2) \in T_{(p,q)}\widetilde{M}_t$ and $\gamma(s) = (p(s),q(s)): I \rightarrow \widetilde{M}_t$ with $\gamma(0) = (p,q)$ and $\gamma'(0) = (v_1,v_2)$ , thus ${}^{\mathbb{L}}\nabla_{(v_1,v_2)}N = \frac{d}{ds}N \circ \gamma(s)\vert_{s=0}$, where ${}^{\mathbb{L}}\nabla$ stands as the connection in the lorentzian space, and $${}^{\mathbb{L}}\nabla_{(v_1,v_2)}N = \frac{1}{\sqrt{2(-1+t^2)}} \left( (v_2,v_1) - t(v_1,v_2)\right), $$ as $\mathbb{H}^2$ is an umbilical hypersurface of $\mathbb{L}^3$ we have the following equation $${}^{\mathbb{L}}\nabla_{(v_1,v_2)}N = {}^{\mathbb{H}}\nabla_{(v_1,v_2)}N + \alpha( (v_1,v_2), N ) = {}^{\mathbb{H}}\nabla_{(v_1,v_2)}N + \frac{1}{\sqrt{2(-1+t^2)}}(\langle v_1, q \rangle p, \langle v_2, p \rangle q ),$$ where ${}^{\mathbb{H}}\nabla$ stands as the connection in the hyperbolic space, and it follows that
$$A(v_1,v_2) = \frac{1}{\sqrt{2\left(-1 + t^2 \right)}} \left(t(v_1,v_2) -(v_2,v_1) + \left(\langle v_1,q \rangle p, \langle v_2,p \rangle q \right) \right).$$

We will need an orthonormal basis to calculate the mean curvature. Let $w \in T\mathbb{H}^2$ with $\langle w,w \rangle_{\mathbb{L}^3} = \frac{1}{2}$ such that $\langle w,p \rangle_{\mathbb{L}^3} = \langle w,q \rangle_{\mathbb{L}^3} = 0$, thus $ \langle (w,-w), N_{(p,q)} \rangle = \langle w, q \rangle - t\langle w, p\rangle - \langle p, w \rangle + t\langle q, w \rangle = 0$ and we have that $(w,-w) \in T\widetilde{M}_t$. Using the same argument, we have that $(w,w) \in T\widetilde{M}_t$. A straightforward calculation shows that $\{(w,-w), (w,w), (q-tp,-p+tq)\}$ is an orthornormal basis of $T\widetilde{M}_t$. Observe that $$A(w,-w) = \frac{1}{\sqrt{2\left(-1 + t^2 \right)}} \left(t(w,-w) +(w,-w)\right) = \frac{1}{\sqrt{2}} \sqrt{\frac{t+1}{t-1}}(w,-w),$$ 
$$A(w,w) = \frac{1}{\sqrt{2}} \sqrt{\frac{t-1}{t+1}}(w,w),$$
$$A(q-tp,-p+tq) = \dfrac{1}{\sqrt{2(-1+t^2)}} \left( (p(1-t^2), q(t^2-1))+((t^2-1)p,(1-t^2)q) \right)=0.$$ It follows that $H_t = \frac{\sqrt{2}t}{3\sqrt{-1+t^2}}$. Observe that the displacement of $M_t$ in direction $N$ at distance $r$ is given by

\begin{equation}
    \begin{split}
    \Phi_r(p,q) &= \exp_{(p,q)}\left(r N_{(p,q)}\right) \\
                &= \left(\left(\cosh \frac{r}{\sqrt{2}}\right)p + \left(\sinh\frac{r}{\sqrt{2}}\right)\frac{q-tp}{\sqrt{-1+t^2}}, \left(\cosh\frac{r}{\sqrt{2}}\right)q + \left(\sinh\frac{r}{\sqrt{2}}\right)\frac{p-tq}{\sqrt{-1+t^2}}\right)\\
                &= \left(\mathcal{P}_r(p,q), \mathcal{Q}_r(p,q)\right).
    \end{split}
\end{equation}

Since $\left\langle\mathcal{P}_r(p,q), \mathcal{Q}_r(p,q) \right\rangle = -t\cosh\left( \sqrt{2}r\right) + \sqrt{-1+t^2}\sinh\left( \sqrt{2}r\right)$, it follows that $\Phi_r (\widetilde{M}_t) = \widetilde{M}_{\phi(r,t)}$, where $\phi(r,t) = t\cosh\left( \sqrt{2}r\right) - \sqrt{-1+t^2}\sinh\left( \sqrt{2}r\right)$.

The MCF with initial data $M_s$ is given by $\Phi_{\epsilon(t)}$, where $\epsilon$ is the solution of the ODE \eqref{ode:MCF}:
\begin{equation}
    \begin{split}
        \epsilon'(t) &= H_{\phi(\epsilon(t),s)}\\
                     &= \frac{\sqrt{2}\phi(\epsilon(t),s)}{3\sqrt{-1+\phi(\epsilon(t),s)^2}}\\
                     &= \frac{\sqrt{2}\left(s\cosh\left(\sqrt{2}\epsilon(t)\right) - \sqrt{-1+s^2}\sinh\left( \sqrt{2}\epsilon(t)\right) \right)}{3\sqrt{-1+\left( s\cosh\left( \sqrt{2}\epsilon(t)\right) - \sqrt{-1+s^2}\sinh\left( \sqrt{2}\epsilon(t)\right) \right)^2}}.
    \end{split}
\end{equation}

%
%
%

\bibliographystyle{abbrv}
\bibliography{bibliography}

\providecommand{\noopsort}[1]{}
\begin{thebibliography}{10}

\bibitem{alexandrino2020}
M.~M. Alexandrino, L.~F. Cavenaghi, and I.~Gon\c{c}alves.
\newblock On mean curvature flow of singular {R}iemannian foliations:
  noncompact cases.
\newblock {\em Differential Geom. Appl.}, 72:101664, 18, 2020.

\bibitem{alexandrino2016}
M.~M. Alexandrino and M.~Radeschi.
\newblock Mean curvature flow of singular {R}iemannian foliations.
\newblock {\em J. Geom. Anal.}, 26(3):2204--2220, 2016.

\bibitem{bookOlmosCia}
J.~Berndt, S.~Console, and C.~E. Olmos.
\newblock {\em Submanifolds and holonomy}.
\newblock Monographs and Research Notes in Mathematics. CRC Press, Boca Raton,
  FL, second edition, 2016.

\bibitem{berndt}
J.~Berndt, F.~Tricerri, and L.~Vanhecke.
\newblock {\em Generalized {H}eisenberg groups and {D}amek-{R}icci harmonic
  spaces}, volume 1598 of {\em Lecture Notes in Mathematics}.
\newblock Springer-Verlag, Berlin, 1995.

\bibitem{isoCartan}
E.~Cartan.
\newblock Familles de surfaces isoparam\'{e}triques dans les espaces \`a
  courbure constante.
\newblock {\em Ann. Mat. Pura Appl.}, 17(1):177--191, 1938.

\bibitem{ChenLeUniqueness}
B.-L. Chen and L.~Yin.
\newblock Uniqueness and pseudolocality theorems of the mean curvature flow.
\newblock {\em Comm. Anal. Geom.}, 15(3):435--490, 2007.

\bibitem{ColdingMinicozziKjaer_mcf}
T.~H. Colding, W.~P. Minicozzi, II, and E.~K.~r. Pedersen.
\newblock Mean curvature flow.
\newblock {\em Bull. Amer. Math. Soc. (N.S.)}, 52(2):297--333, 2015.

\bibitem{daniel}
B.~Daniel.
\newblock Isometric immersions into 3-dimensional homogeneous manifolds.
\newblock {\em Comment. Math. Helv.}, 82(1):87--131, 2007.

\bibitem{DiazVazquezCastineiraDef}
J.~C. D\'{\i}az-Ramos, M.~Dom\'{\i}nguez-V\'{a}zquez, and
  C.~Vidal-Casti\~{n}eira.
\newblock Isoparametric submanifolds in two-dimensional complex space forms.
\newblock {\em Ann. Global Anal. Geom.}, 53(2):205--216, 2018.

\bibitem{notas-miguel}
M.~Dom\'inguez-V\'azquez.
\newblock {\em An introduction to isoparametric foliations}.
\newblock Preprint. Available at
  http://xtsunxet.usc.es/miguel/teaching/jae2018.html, 2018.

\bibitem{dominguez-manzano}
M.~Dom\'{\i}nguez-V\'{a}zquez and J.~M. Manzano.
\newblock Isoparametric surfaces in {$\Bbb E(\kappa,\tau)$}-spaces.
\newblock {\em Ann. Sc. Norm. Super. Pisa Cl. Sci. (5)}, 22(1):269--285, 2021.

\bibitem{freitas-abdenago}
A.~A. Freitas.
\newblock {\em Fluxos geom\'etricos por hipersuperf\'icies paralelas (Geometric
  flows through parallel hypersurfaces)}.
\newblock Thesis (PhD in Mathematics).54f. Centro de Ci\^encias, Universidade
  Federal do Cear\'a, Fortaleza, Brasil, 2021.

\bibitem{graysonShort}
M.~A. Grayson.
\newblock Shortening embedded curves.
\newblock {\em Ann. of Math. (2)}, 129(1):71--111, 1989.

\bibitem{HeintzeLiuOlmosDefIso}
E.~Heintze, X.~Liu, and C.~Olmos.
\newblock Isoparametric submanifolds and a {C}hevalley-type restriction
  theorem.
\newblock In {\em Integrable systems, geometry, and topology}, volume~36 of
  {\em AMS/IP Stud. Adv. Math.}, pages 151--190. Amer. Math. Soc., Providence,
  RI, 2006.

\bibitem{HuiskenConvexHyper}
G.~Huisken.
\newblock Contracting convex hypersurfaces in {R}iemannian manifolds by their
  mean curvature.
\newblock {\em Invent. Math.}, 84(3):463--480, 1986.

\bibitem{Isenberg2019}
J.~Isenberg and H.~Wu.
\newblock Mean curvature flow of noncompact hypersurfaces with {T}ype-{II}
  curvature blow-up.
\newblock {\em J. Reine Angew. Math.}, 754:225--251, 2019.

\bibitem{Isenberg2020}
J.~Isenberg, H.~Wu, and Z.~Zhang.
\newblock Mean curvature flow of noncompact hypersurfaces with {T}ype-{II}
  curvature blow-up. {II}.
\newblock {\em Adv. Math.}, 367:107111, 44, 2020.

\bibitem{s2r2Batalla}
J.~Julio-Batalla.
\newblock Isoparametric functions on {$\Bbb R^n\times\Bbb M^m$}.
\newblock {\em Diff. Geom. and its Appl.}, 60(6):1--8, 2018.

\bibitem{levi}
T.~Levi-Civita.
\newblock Famiglie di superficie isoparametriche nell'ordinario spazio
  euclideo.
\newblock {\em Accad. Naz. Lincei Rend. Cl. Sci. Fis. Mat. Natur.},
  26(6):355--362, 1937.

\bibitem{ronaldo}
{\noopsort{Lima}}{R. F. de Lima}.
\newblock Weingarten flows in {R}iemannian manifolds.
\newblock arXiv:2205.09566, 2022.

\bibitem{radeschi-mcf}
X.~Liu and M.~Radeschi.
\newblock Polar foliations on symmetric spaces and mean curvature flow.
\newblock {\em to appear in J. Reine Ang. Math. (Crelle)}, arXiv:2006.03945,
  2022.

\bibitem{liu-terng-1}
X.~Liu and C.-L. Terng.
\newblock The mean curvature flow for isoparametric submanifolds.
\newblock {\em Duke Math. J.}, 147(1):157--179, 2009.

\bibitem{liu-terng-2}
X.~Liu and C.-L. Terng.
\newblock Ancient solutions to mean curvature flow for isoparametric
  submanifolds.
\newblock {\em Math. Ann.}, 378(1-2):289--315, 2020.

\bibitem{RVazquezExIso}
A.~Rodr\'{\i}guez-V\'{a}zquez.
\newblock A nonisoparametric hypersurface with constant principal curvatures.
\newblock {\em Proc. Amer. Math. Soc.}, 147(12):5417--5420, 2019.

\bibitem{reis-tenenblat}
H.~F. Santos~dos Reis and K.~Tenenblat.
\newblock The mean curvature flow by parallel hypersurfaces.
\newblock {\em Proc. Amer. Math. Soc.}, 146(11):4867--4878, 2018.

\bibitem{segre}
B.~Segre.
\newblock Famiglie di ipersuperficie isoparametriche negli spazi euclidei ad un
  qualunque numero di dimensioni.
\newblock {\em Accad. Naz. Lincei Rend. Cl. Sci. Fis. Mat. Natur.},
  27(6):203--207, 1938.

\bibitem{surveySmoczyk}
K.~Smoczyk.
\newblock Mean curvature flow in higher codimension: introduction and survey.
\newblock In {\em Global differential geometry}, volume~17 of {\em Springer
  Proc. Math.}, pages 231--274. Springer, Heidelberg, 2012.

\bibitem{somigliana}
C.~Somigliana.
\newblock Sulle relazioni fra il principio di huygens e l'ottica geometrica.
\newblock {\em Atti Acc. Sc. Torino}, LIV:974--979, 1918-1919.

\bibitem{s2s2Urbano}
F.~Urbano.
\newblock On hypersurfaces of {$\Bbb S^2\times\Bbb S^2$}.
\newblock {\em Comm. Anal. Geom.}, 27(6):1381--1416, 2019.

\bibitem{WangExIso}
Q.~M. Wang.
\newblock Isoparametric hypersurfaces in complex projective spaces.
\newblock In {\em Proceedings of the 1980 {B}eijing {S}ymposium on
  {D}ifferential {G}eometry and {D}ifferential {E}quations, {V}ol. 1, 2, 3
  ({B}eijing, 1980)}, pages 1509--1523. Sci. Press Beijing, Beijing, 1982.

\end{thebibliography}
\end{document}